\documentclass[11pt]{amsart}
\usepackage{amsmath}
\usepackage{amssymb}
\usepackage{mathrsfs}
\usepackage{amscd}
\theoremstyle{plain}
    \newtheorem{theorem}{Theorem}[section]
    \newtheorem{proposition}[theorem]{Proposition}
    \newtheorem{lemma}[theorem]{Lemma}
    \newtheorem{corollary}[theorem]{Corollary}
     
\theoremstyle{definition}

\newcommand{\bbZ}{\mathbb{Z}}
\newcommand{\bbC}{\mathbb{C}}
\newcommand{\bbG}{\mathbb{G}}
\newcommand{\bbQ}{\mathbb{Q}}

\newcommand{\cO}{\mathcal{O}}
\newcommand{\bsK}{{\boldsymbol{K}}}

\begin{document}
\title[$p$-adic Fourier theory]{Integral structures on $p$-adic Fourier theory}
\author[Bannai]{Kenichi Bannai}
\address{Department of Mathematics, Keio University, 3-14-1 Hiyoshi, Kouhoku-ku, Yokohama 223-8522, Japan}
\email{bannai@math.keio.ac.jp}
\author[Kobayashi]{Shinichi Kobayashi}
\address{Mathematical Institute, Tohoku University, 6-3 Aramaki-aza-Aoba, Aoba-ku, Sendai  980-8578, Japan}
\email{shinichi@math.tohoku.ac.jp}
\date{March 27, 2015}

\begin{abstract}  
 	In this article, we give an explicit construction of the $p$-adic Fourier transform by Schneider and Teitelbaum,
	which allows for the investigation of the integral property.
	As an application, we give a certain  integral basis of  the space of $K$-locally analytic  functions 
	on the ring of integers $\cO_K$ for any finite extension $K$ of $\bbQ_p$, generalizing the basis constructed by
	Amice for locally analytic functions on $\bbZ_p$.  
	We also use our result to prove congruences of Bernoulli-Hurwitz numbers at non-ordinary (i.e. supersingular) primes 
	originally investigated by Katz and Chellali. 
\end{abstract}

\thanks{The first and second authors were supported by the JSPS Postdoctoral Fellowships for Research Abroad 2005-2007/2004-2006. 
This work was also supported in part by KAKENHI (21674001, 25707001, 26247004)}


\maketitle

\section{Introduction}

One important method in studying the congruences  and $p$-adic properties
of important invariants in number theory is the use of $p$-adic measures interpolating such values.
Such theory was applied to obtain the Kummer congruence between special values of Riemann zeta
function as well as the construction of the $p$-adic $L$-functions for elliptic curves
with ordinary reduction at $p$.   When dealing with the non-ordinary case,
it is necessary to use the theory of $p$-adic analytic distributions, which is a generalization
of the theory of $p$-adic measures.  
For such $p$-adic distributions on $\bbZ_p$, the Amice transform gives a 
one-to-one correspondence between
$\bbC_p$-valued distributions on $\bbZ_p$ and rigid analytic functions on the
open unit disc.    The general idea is to study the congruences and $p$-adic properties of the interpolated
invariants through the $p$-adic property of the rigid analytic function corresponding 
to the $p$-adic distribution.  However, contrary to the case of $p$-adic measures,
the Amice transform is not well-behaved integrally for general $p$-adic distributions,
hence it is necessary to investigate in detail the precise integral structure of this transform.
Amice \cite[\S 10]{Am} investigated the precise integral structure of the Amice transform. 

Let $\cO_K$ be the ring of integers of a finite extension $K$ of $\bbQ_p$.  
In \cite[\S 4]{ST}, Schneider and Teitelbaum constructed the  $p$-adic Fourier transform,
which is a one-to-one correspondence between $\bbC_p$-valued
distributions on $\cO_K$ and rigid analytic functions on an open unit disc.
The purpose of this article is to give an explicit and elementary construction of the $p$-adic Fourier transform
of Schneider-Teitelbaum, which allows investigation of the precise integral structure of this correspondence.
We then determine an integral structure on the ring of locally analytic functions on $\cO_K$.
The integrality of the $p$-adic Fourier transform for general $K$ is even less well behaved than for the
case of $\bbQ_p$; even if the
rigid analytic function corresponding to a $p$-adic distribution has bounded coefficients,
the $p$-adic distribution may not necessarily be a $p$-adic measure.
As an application of our result, we obtain  the congruences originally proved by Katz 
\cite[Theorem 3.11]{Ka2} and Chellali \cite[Th\'eor\`em 1.1]{Ch}
of Bernoulli-Hurwitz numbers, which are essentially special values of $p$-adic $L$-functions of CM elliptic curves
at non-ordinary primes.


We now give the exact statements of our theorems. 
Let $p$ be a rational prime and let $|\cdot|$ be the absolute value of $\bbC_p$ such that $|p|=p^{-1}$. 
Let $\pi$ be an uniformizer of $\cO_K$, and let $\mathbb{F}_q$ be the residue field of $\cO_K$.
We define $LA_N(\cO_K, \mathbb{C}_{p})$ to be the  space  of  locally analytic  functions on 
$\cO_K$ of order $N$ which take values in $\mathbb{C}_{p}$.  
Namely, $f(x) \in LA_N(\cO_K,\mathbb{C}_{p})$ if and only if $f(x)$ is defined as 
a power series $\sum_{n=0}^\infty a_n(x-a)^n$ on $a+\pi^N\cO_K$ for any $a \in \cO_K$. 
We let $\Vert f \Vert_{a,N}:=\max_n\{|a_n \pi^{nN}| \}$. 
The space $LA_N(\cO_K, \mathbb{C}_{p})$ is a $p$-adic 
Banach space induced by the norm
 $\max_{a \in \cO_K}\{\Vert f \Vert_{a,N}\}$ and we denote by 
 $LA_N(\cO_K, \mathbb{C}_{p})_{0}$  the submodule of elements whose 
 absolute values are less than or equal to $1$.
We let $\mathcal{G}$ be a Lubin-Tate group of $K$ corresponding to $\pi$, and
let $\varpi_p \in \bbC_p$ be a $p$-adic period of $\mathcal{G}$. 
We let  
$$\overline{\rho}(k)=\max_{k \leq m}\{|{m!}/{\varpi_p^m}|\}, 
\quad 
\underline{\rho}(k)
=\min_{0 \leq m \leq k}\{{|m!}/{\varpi_p^m}|\}.$$

Let $\varphi(t)$ be a rigid analytic function on the open unit disc.
In other words, $\varphi(t)$ is a power series of the form $\varphi(t)=\sum_{n=0}^\infty c_n t^n$ such that 
$|c_n| r_0^n\rightarrow 0$ for any  $0<r_0<1$. 
Let $\mu_{\varphi}$ be the distribution on $\cO_K$ corresponding  to $\varphi(t)$ 
given by Schneider-Teitelbaum's $p$-adic Fourier theory \cite[Theorem 2.3]{ST}.  
Then we have the following

\begin{theorem}\label{theorem; intro}
Let $f$ be a $K$-locally analytic function in  $LA_N(\cO_{K}, {\bbC_p})$. 
Then we have 
\begin{align}
\left|\int_{a+\pi^N\cO_K} \; f(x) \;d\mu_{\varphi} \right| \; \leq \; 
 \overline{\rho}(0) 
 \left| \frac{\pi}{q}\right|^{N}  \;\Vert f \Vert_{a,N} 
 \Vert \varphi\Vert_{N}
\end{align}
 where \begin{equation}
\Vert \varphi\Vert_N:=\max_{k}
\left\{\; |c_k| \overline{\rho}\left(\left[\frac{k}{q^N}\right] \right) \; \right\}
\end{equation} 
and  $[x]$ is the integral part of $x$.
\end{theorem}

The crucial difference from the case when $K=\bbQ_p$ is the fact that $|\pi/q|>1$ when $K \neq \bbQ_p$.   
A finer version of the above is given as Theorem \ref{theorem; main}. 
Since $\overline{\rho}\left(\left[\frac{k}{q^N}\right] \right) \sim p^{-kr}$ where 
$r=1/eq^N(q-1)$, the value 
$\Vert \varphi\Vert_N$ is approximated  by 
$$ \Vert \varphi\Vert_{\overline{\bold{B}}(p^{-r})}=\max_{x \in \overline{\bold{B}}(p^{-r})}\{\;|\varphi(x)|\;\}$$ 
where $\overline{\bold{B}}(p^{-r}) \subset \bbC_p$ is the closed disc of radius $p^{-r}$ centered  at the origin.  

%

As an application of our main theorem,  we 
obtain  an estimate of the Fourier coefficients of Mahler like expansion of 
functions in $LA_N(\cO_{K}, \bbC_p)$. 
Let $\lambda(t)$ be the formal logarithm of  $\mathcal{G}$, and following \cite{ST},
we define the polynomial $P_n(x)$ by  
$$\exp(x \lambda(t))=\sum_{n=0}^\infty \; P_n(x) t^n.$$
Note that when $\mathcal{G}$ is the multiplicative formal group $\mathcal{G} = \widehat\bbG_m$, then $\lambda(t) = 1+t$ and the above expansion 
is simply
$$
	(1+t)^x = \sum_{n=0}^\infty \binom{x}{n} t^n.
$$
Hence the polynomial $P_n(t)$ is the generalization of the binomial polynomial
$$
	\binom{x}{n}  = \frac{x (x-1) \cdots (x-n+1)}{n!}.
$$
Then we have the following.

\begin{theorem}(=Theorem \ref{theorem; mahler L_N}.)
The series $\sum_{n=0}^\infty a_n P_n(x \varpi_p)$ converges to 
an element of  $LA_N(\cO_{K}, \mathbb{C}_p)_{0}$ for $a_n$  satisfying
\begin{align*}
	|a_n| & \leq \underline{\rho}\left(\left[ \frac{n}{q^N}\right]\right), &
	\lim_{n \rightarrow 0} &|a_n|/\underline{\rho}\left(\left[ \frac{n}{q^N}\right]\right)= 0.
\end{align*}
Conversely, if $f(x) \in LA_N(\cO_{K}, \mathbb{C}_p)_{0}$, then it has 
an expansion 
$$f(x)=\sum_{n=0}^\infty a_n P_n(x \varpi_p)$$ 
of the form 
\begin{align*}
	|a_n| & \leq c \left |\frac{\pi}{q} \right |^N  \overline{\rho}\left(\left[ \frac{n}{q^N}\right]\right), &
	\lim_{n \rightarrow 0} &|a_n|/\overline{\rho}\left(\left[ \frac{n}{q^N}\right]\right)= 0,
\end{align*}
where $ c=1$ if $e \leq p-1$,  and  $c =\overline{\rho}(0) $, otherwise. 
\end{theorem}

\begin{corollary}(=Corollary \ref{corollary; basis}.)
Suppose
$$e_{N,n}:=\underline{\gamma}\left(\left[\frac{n}{q^N}\right] \right) P_n(x \varpi_p), \qquad (n=0,1, \cdots),$$  
where $\underline{\gamma}(u)$ is an element in $\bbC_p$ such that $\underline{\rho}(u) = |\underline{\gamma}(u)|$.
We denote  by $L_N$  the $\cO_{\bbC_p}$-module topologically generated by 
$e_{N,n}$, then 
$$\overline{\rho}(0)^{-2}\left |\frac{q}{\pi} \right |^N LA_N(\cO_K, {\bbC_p})_{0}\;\subset\; L_N
\; \subset \; LA_N(\cO_K, {\bbC_p})_{0}.$$
In particular, 
$L_N\otimes \bbQ_p=LA_N(\cO_K, {\bbC_p})$, namely, 
the functions $e_{N,n}$ form a Banach basis of $LA_N(\cO_K, {\bbC_p})$. 
Moreover, if $e \leq p-1$, then 
$$\left |\frac{q}{\pi} \right |^{N+1} LA_N(\cO_K, {\bbC_p})_{0}\;\subset\; L_N
\; \subset \; LA_N(\cO_K, {\bbC_p})_{0}.$$
In particular, 
If $\cO_K=\bbZ_p$, we recover Amice's result \cite[Th\'eor\`em 3]{Am}, namely,  
$$\left[\frac{n}{p^N}\right]! \binom{x}{n}, \qquad (n=0,1, \cdots)$$ form 
a topological basis of  $LA_N(\bbZ_p, {\bbC_p})_{0}$. 
(Actually, we can show that it is a basis of $LA_N(\bbZ_p, \bbQ_p)_{0}$.)
\end{corollary}

As another application, in Theorem \ref{theorem, KC}, we derive from 
our estimate of the integral the congruence of 
Bernoulli-Hurwitz numbers  $BH(n)$ at supersingular primes established by Katz 
\cite[Theorem 3.11]{Ka2} and Chellali \cite[Th\'eor\`em 1.1]{Ch}.

\subsection{Acknowledgment}

Part of this research was conducted while the first author was visiting the Ecole Normale Sup\'erieure in Paris,
and the second author Institut de Math\'ematiques de Jussieu.  The authors would like to thank their
hosts Yves Andr\'e and Pierre Colmez for hospitality.  The authors would also like to thank 
 Seidai Yasuda for reading an earlier version of the manuscript and pointing out mistakes.

\section{Schneider-Teitelbaum's $p$-adic Fourier theory.}  

Let $K$ be a finite extension of $\bbQ_p$ and 
$k=\mathbb{F}_q$ the residue field.  
Let $e$ be the absolute ramification index of $K$. 
We fix a uniformaizer $\pi$ of $K$ and 
let $\mathcal{G}$ be the Lubin-Tate formal group of $K$ 
associated to $\pi$. 
For a natural number $N$ and an element  $a$ of $\cO_{K}$, 
we define the space $A(a+\pi^N\cO_{K}, \bbC_p)$
 of $K$-analytic functions on $a+\pi^N\cO_{K}$ as follows.  
 $$
\{ f:a+\pi^N\cO_{K} \rightarrow \bbC_p \;|\; f(x)=\sum_{n=0}^\infty a_n (x-a)^n, \,
a_n \in \bbC_p,\, \pi^{nN}a_n \rightarrow 0 
 \}.$$
We equip  the space $A(a+\pi^N\cO_{K}, \bbC_p)$ with the norm 
$$
\Vert f \Vert_{a,N}:=\mathrm{max}_n \; \{|\pi^{nN}a_n |\}
=\mathrm{max}_{x \in a+\pi^N\cO_{\bbC_p}} \; \{|f(x) |\}.$$
We also define  the space 
$LA_N(\cO_{K}, \bbC_p)$ of  locally $K$-analytic functions on $\cO_{K}$ of 
order $N$ 
by 
 $$
\{ f:\cO_{K} \rightarrow \bbC_p \;|\;\; f \,|_{a+\pi^N\cO_{K}} \in A(a+\pi^N\cO_{K}, \bbC_p)\quad  \text{for any} \; a \in \cO_{K} 
 \}$$
which is a Banach space by the norm $\mathrm{max}_a \; \{\Vert f \Vert_{a, N}\}$. 
We denote by 
 $LA_N(\cO_K, \mathbb{C}_{p})_{0}$  the submodule of elements whose 
 absolute values are less than or equal to $1$.
We put $$LA(\cO_{K}, \bbC_p)=\bigcup_N LA_N(\cO_{K}, \bbC_p)$$
and equip it with the inductive limit topology. 
A  continuous $\bbC_p$-linear function $LA(\cO_{K}, \bbC_p) 
\rightarrow \bbC_p$ is called an $\bbC_p$-valued distribution 
on $\cO_{K}$. 
We denote the space of $\bbC_p$-valued distributions 
on $\cO_{K}$ by $D(\cO_{K}, \bbC_p)$, namely,  
$$D(\cO_{K}, \bbC_p)= \varprojlim_N \mathrm{Hom}_{\bbC_p}^{\mathrm{cont}} (LA_N(\cO_{K}, \bbC_p), \bbC_p).$$ 
We write an element of $D(\cO_{K}, \bbC_p)$ symbolically as 
$$\int d\mu: LA(\cO_{K}, \bbC_p) 
\rightarrow \bbC_p, \quad f \mapsto \int f d\mu=\int_{\cO_{K}} f (x)\, d\mu(x).$$
The space  $D(\cO_{K}, \bbC_p)$ has  a product structure given by the convolution product.
For a compact open set $U$ of $\cO_K$, we let 
$$\int_{U} f (x)d\mu(x):=\int_{\cO_{K}} f (x) \cdot 1_U(x) \,d\mu(x)$$
where $1_U$ is the characteristic function of $U$.  

The structure of 
$D(\cO_{K}, \bbC_p)$ is well-known for the case $K=\bbQ_p$ and described through 
 the so called  Amice transform. 
We denote by $R^{\mathrm{rig}}$ 
 the ring of 
rigid analytic functions on the open disc of radius $1$, 
namely, the ring of power series of the form 
$\varphi(T)=\sum_{n=0}^\infty c_n T^n$ such that 
$|c_n| r_0^n\rightarrow 0$ for any  $0<r_0<1$. 
Then there exists an isomorphism of topological $\bbC_p$-algebras  
\begin{equation}\label{equation; structure of distribution}
 D({\bbZ_p}, \bbC_p)\cong R^{\mathrm{rig}}, \qquad \mu \mapsto \varphi
\end{equation}
that is characterized by the equation 
$$c_n=\int_{\bbZ_p} \binom{x}{n} d\mu(x)$$ 
or equivalently 
$$\varphi (T)=\int_{\bbZ_p} (1+T)^x\, d\mu(x).$$
For the Mahler expansion
$$f(x)=\sum_{n=0}^\infty a_n \binom{x}{n}$$
of $f \in LA(\bbZ_p, \bbC_p)$, Amice showed that 
$|a_n| r^n \rightarrow 0$ for some $r>1$ and  
hence we can compute the integral as 
\begin{equation}\label{equation: amice}
\int_{\bbZ_p} f(x)\, d\mu=\sum_{n=0}^\infty a_n c_n.
\end{equation}
Schneider-Teitelbaum \cite[Theorem 2.3]{ST} constructed an
 isomorphism analogous with (\ref{equation; structure of distribution}) for 
a general local field $K$. 

Let $\varpi_p$ be a $p$-adic period of $\mathcal{G}$. 
Namely,  by Tate's theory of $p$-divisible groups and the Lubin-Tate theory we have
$$\mathrm{Hom}_{\cO_{\bbC_p}}(\mathcal{G}, \widehat{\bbG}_m ) \cong 
\mathrm{Hom}_{\bbZ_p}(T_p\mathcal{G}, T_p\widehat{\bbG}_m )\cong \cO_{K}.$$ 
(The last isomorphism is non-canonical.)
Hence there exists a generator of the
$\cO_K$-module $\mathrm{Hom}_{\cO_{\bbC_p}}(\mathcal{G}, \widehat{\bbG}_m )$, which 
is written in the form of the integral power series $\exp(\varpi_p \lambda(t)) \in \cO_{\bbC_p}[[t]]$
where $\lambda(t)$  is the logarithm of  $\mathcal{G}$. 
The element 
$\varpi_p \in {\cO_{\bbC_p}}$ is determined uniquely up to an element of $\cO_{K}^\times$. 
We fix such a $\varpi_p$ and  call it the $p$-adic period  of $\mathcal{G}$. 
(If the height of $\mathcal{G}$ is equal to $1$, the inverse of $\varpi_p$ is often called 
a $p$-adic period of  $\mathcal{G}$, for example, see \cite{dS}). 
It is known that $|\varpi_p|=p^{-s}$, where 
$s=\frac{1}{p-1}-\frac{1}{e(q-1)}$ (see  Appendix of \cite{ST} or 
an elementary proof in \cite{Box1} when $K/\bbQ_p$ is unramified). 
We define the polynomials $P_n(X) \in K[X]$  by the formal expansion   
$$\exp(X\lambda(t))=\sum_{n=0}^\infty P_n(X)\, t^n.$$
Note that in the case $\mathcal{G}=\widehat{\bbG}_m$, $\pi=p$ and 
$\lambda(t)=\log(1+t)$,  
the polynomial $P_n(X)$ is no other than the binomial polynomial  
$\binom{X}{n}$.  
By construction, $P_n(x\varpi_p)$ is in $\cO_{\bbC_p}$
if $x \in \cO_{K}$.


\begin{theorem}$($Schneider-Teitelbaum \cite[\S 4]{ST}$)$:
i) The series $$\sum_{n=0}^\infty a_n P_n(x \varpi_p)$$ converges to 
an element of  $LA(\cO_{K}, \bbC_p)$ if $\overline{\lim}_n \,|a_n|^{\frac{1}{n}}<1$. 
Conversely, any locally $K$-analytic function $f(x)$ on $\cO_{K}$ has a unique  expansion  
$$f(x)=\sum_{n=0}^\infty a_n P_n(x \varpi_p)$$ 
for some sequence $(a_n)_n$ in $\bbC_p$ such that  $\overline{\lim}_n \,|a_n|^{\frac{1}{n}}<1$. \\
ii) There exists an isomorphism of topological $\bbC_p$-algebras  
\begin{equation}\label{equation; structure of distribution K}
 D(\cO_{K}, \bbC_p)\cong R^{\mathrm{rig}}.
\end{equation}
having the following characterization property: if $\varphi(T)=\sum_{n=0}^\infty c_n T^n$ 
corresponds to a distribution $\mu$, 
then 
$$c_n=\int_{\cO_{K}} P_n(x\varpi_p) \, d\mu(x)$$ 
or equivalently 
$$\varphi (t)=\int_{\cO_{K}} \exp(x \varpi_p\lambda(t)) \, d\mu(x).$$
Schneider and Teitelbaum called the power series $\varphi(t)$ corresponding to $\mu$ the 
Fourier transform of $\mu$ and denoted it by $F_\mu(t)$.
\end{theorem}


\section{Power sums}\label{Power sums}

In this section, we give an estimate of the absolute value of 
the power sum 
$$	S_{N,n,k}:=
	\partial^n_\mathcal{G} 
	\sum_{t_N \in \mathcal{G}[\pi^N]} (t \oplus t_N)^k |_{t=0},
$$
where $x \oplus y = \mathcal{G}(x,y)$, $\partial_\mathcal{G}$ is the differential operator $ \lambda'(t)^{-1} (d/dt)$, and 
$\mathcal{G}[\pi^N]$ is the kernel of the multiplication $[\pi^N]$ of $\mathcal{G}$. 
This estimate is crucial for everything in this paper. 
We use Newton's method to compute this value. 

We define $\overline{\rho}[l,n]$ and $\underline{\rho}[l,n]$ by
$$\overline{\rho}[l,n]=\max_{l \leq m \leq n }\{|{m!}/{\varpi_p^m}|\}, \qquad 
\underline{\rho}[l,n]=\min_{l \leq m \leq n}\{|{m!}/{\varpi_p^m}|\}$$
for $l \leq n$. 
For $l >n$, we put $\overline{\rho}[l,n]=0$ and 
$\underline{\rho}[l,n]=\infty$. 
Then $\overline{\rho}(k)=\overline{\rho}[k,\infty]$ 
and $\underline{\rho}(k)=\underline{\rho}[0,k]$ are as in the introduction. 

\begin{proposition}\label{prop, gamma estimate}
i) The values $\overline{\rho}(k)$ and $\underline{\rho}(k)$ are 
decreasing with $k$. \\
ii) We have 
$$ \underline{\rho}(k) \leq \overline{\rho}(k), \qquad  \overline{\rho}(k) \leq 
\overline{\rho}(0)   \underline{\rho}(k).$$
iii) We have 
$$\underline{\rho}(k_1+\cdots+k_n)
\leq 
\underline{\rho}(k_1)\cdots \underline{\rho}(k_n). 
$$
iv) We have $$
p^{\frac{1}{p-1}-\frac{k}{e(q-1)}} \;\leq\; \underline{\rho}(k) \leq  1. 
$$
\end{proposition}
\begin{proof}
i) is clear. For ii), first we have $\overline{\rho}(k) \geq 
|k!/\varpi_p^k |\geq \underline{\rho}(k)$. 
Suppose $\overline{\rho}(k)=| k_1!/\varpi_p^{k_1}|$ and $\underline{\rho}(k)
=|k_2!/\varpi_p^{k_2}|$. Then $k_1 \geq k \geq k_2$ and 
$$\left|\frac{k_1!}{\varpi_p^{k_1}}/\frac{k_2!}{\varpi_p^{k_2}}\right|
=\left|\binom{k_1}{k_2} \frac{(k_1-k_2)!}{\varpi_p^{k_1-k_2}}\right| \leq \overline{\rho}(0).$$
For iii), suppose that 
$\underline{\rho}(k_i)
=| l_i!/\varpi_p^{l_i}| $ for $l_i \leq k_i$. Then the assertion for $\underline{\rho}$ follows from 
$$\underline{\rho}(k_1+\dots+k_n) \leq \left|
\frac{(l_1+\dots +l_n)!}{\varpi_p^{l_1+\dots +l_n}}\right| 
\leq \left|\frac{(l_1+\dots +l_n)!}{l_1! \cdots l_n!}\right|
\left|
\frac{l_1!}{\varpi_p^{l_1}}\right| \cdots 
\left|
\frac{l_n!}{\varpi_p^{l_n}}\right|.
$$

For iv), suppose that $\underline{\rho}(k)
= | l!/\varpi_p^{l}| $ for $l \leq k$. Then 
$$
p^{\frac{1}{p-1}-\frac{k}{e(q-1)}} \;\leq\;
p^{\frac{1}{p-1}-\frac{l}{e(q-1)}} \;\leq\;
\left| \frac{l!}{\varpi_p^l}\right|=
 \underline{\rho}(k). $$
\end{proof}

If $e \leq p -1$, then we can determine $\overline{\rho}(k)$ and 
$\underline{\rho}(k)$ explicitly. 

\begin{lemma}\label{lemma; binom}
Let $k$ be a non-negative integer and let $q$ be a power of $p$. \\
i) For any integer   $0 \leq r <q$, we have 
$\binom{kq+r}{r}  \equiv 1 \mod p.$ \\
ii) We have $\binom{k}{q}  \in [k/q]\bbZ_p$. 
\end{lemma}
\begin{proof}
i) is clear. For ii), we write $k=aq+r$ with $0 \leq r<q$. 
We put $(1+x)^q=1+x^q+pf(x)$ for some integral polynomial $f(x)$. 
Then 
$$(1+x)^{k}=
(1+x^q+pf(x))^{a}(1+x)^r \equiv  (1+x^q)^{a}(1+x)^r \mod ap \bbZ_p[x].$$
Hence the coefficient of $x^q$ in the above is in $a \bbZ_p$. 
\end{proof}

\begin{proposition}\label{proposition; factorial/varpi}
 Let $i$, $e$ and $h$ be natural numbers. We put $q=p^h$. 
Then we have 
$$v_p(i!) \geq \frac{i}{p-1}-\frac{i}{e(q-1)}-h+\frac{1}{e}
+\left[\frac{i}{q}\right]\left(\frac{1}{e}-\frac{1}{p-1}+\frac{1}{e(q-1)}\right)
+v_p\left(\left[\frac{i}{q}\right]! \right)
.$$
In the above, the equality holds if and only if $i \equiv -1 \mod q$.
In particular, if $e \leq p-1$ or $i<q$, we have 
$$v_p(i!) \geq \frac{i}{p-1}-\frac{i}{e(q-1)}-h+\frac{1}{e}$$ 
and the equality holds if and only if $i=q-1$, and we have 
$\overline{\rho}(0)=|\pi/q|$. 
\end{proposition}
\begin{proof}
First, we assume  that $i<q$. We prove the inequality by induction on $h$. 
If $h=1$, then $i<p$. Hence the left hand side is equal to zero, namely  
 $v_p(i!)=0$, and  the right hand side take the maximum value when $i=p-1$, which 
is also equal to zero.  We assume that the inequality holds for natural numbers less than $h$. 
Since the right hand side is strictly increasing for $i$, and $v_p(i!)$  strictly increase 
only when $p$ divides $i$, we may assume that $i$ is of the form $i=kp-1$ 
for some natural number $k \leq p^{h-1}$. We have 
$$v_p(i!)=v_p((kp)!)-v_p(kp)=k-1+v_p((k-1)!).$$ On the other hand, we have 
\begin{align*}
&\frac{i}{p-1}-\frac{i}{e(q-1)}-h+\frac{1}{e}\\
&=(k-1)+\frac{k-1}{p-1}-\frac{k-1}{e(p^{h-1}-1)}-(h-1)+\frac{1}{e}+\frac{k-1}{e(p^{h-1}-1)}-\frac{kp-1}{e(q-1)}\\
& \leq k-1+v_p((k-1)!).
\end{align*}
In the last inequality, we used the inductive hypothesis and $k \leq p^{h-1}$. 
Hence we have the desire inequality and the equality holds only when $k=p^{h-1}$, namely, 
$i=q-1$. 
For $i  \geq q$, by Lemma \ref{lemma; binom} ii) and  induction, we have 
\begin{align*}
&v_p(i!) \geq v_p((i-q)!)+v_p(q!)+v_p\left(\left[\frac{i}{q}\right] \right) \\
&\geq \frac{i}{p-1}-\frac{i}{e(q-1)}-h+\frac{1}{e}+\left[\frac{i}{q}\right]\left(\frac{1}{e}-\frac{1}{p-1}+\frac{1}{e(q-1)}\right)+v_p\left(\left[\frac{i}{q}\right]! \right). 
\end{align*} 
From the above argument and the induction,  to have the equality, $i$ must be congruent to $-1$. 
On the other hand, if $i \equiv -1 \mod q$, then   direct calculations give the equality.  
\end{proof}

\begin{proposition}\label{proposition; gamma(n)}
Suppose that 
 $e \leq p-1$, and that $e >1$ or $h>1$. \\
i) We have $|n!/ \varpi_p^n| > 1$ for $0<n<q$. \\
ii) For any non-negative integer $n$, $\underline{\rho}(n)=|n_0!/\varpi_p^{n_0}|$ 
where $n_0=[n/q]q$. \\
iii) For $n \equiv -1 \mod q$ and a natural number $i \not=q$, we have  
$$\left|\frac{n!}{ \varpi_p^n}\right|>\left|\frac{(n+q)!}{ \varpi_p^{n+q}}\right|
>\left|\frac{(n+i)!}{ \varpi_p^{n+i}}\right|$$
In particular, for any non-negative integer $n$, we have $\overline{\rho}(n)=|n_1!/\varpi_p^{n_1}|$ 
where $n_1=[n/q]q+q-1$. 
\end{proposition}
\begin{proof} 
We prove i) by induction for $h$ of $q=p^h$. 
If $h=1$, then $n!$ is a $p$-adic unit and the assertion is clear. 
Assume that $h>1$. We write as $n=kp+r$ with $0 \leq r <p$. 
Then 
$$\frac{n!}{\varpi_{p}^n}
=\binom{n}{r}\frac{(kp)!}{\varpi_p^{kp}}\frac{r!}{\varpi_p^r}.$$
Hence by Lemma \ref{lemma; binom} i) and the induction for $n$, 
we may assume that $r=0$ and $k \geq 1$. 
Then 
$$v_p\left(\frac{(kp)!}{\varpi_p^{kp}}\right)=v_p((kp)!)-\frac{kp}{p-1}+\frac{kp}{e(q-1)}
< v_p(k!)-\frac{k}{p-1}+\frac{k}{e(p^{h-1}-1)}. 
$$
By the inductive hypothesis for $h$, the right hand side is negative or $0$. 

Next we prove ii).  Suppose that 
$m <n_0$. 
Then 
$$\left|\frac{n_0!}{\varpi^{n_0}}/
\frac{m!}{\varpi_p^m}\right|
=\left|
\frac{n_0}{\varpi_p}\binom{n_0-1}{m}\frac{(n_0-m-1)!}{\varpi_p^{n_0-m-1}} 
\right| \leq \left| \frac{n_0}{\varpi_p} \right| \overline{\rho}(0)
= \left| \frac{n_0 \pi}{q\varpi_p} \right|<1.$$
Suppose that $n \geq m> n_0$. 
We write as $m=[n/q]q+r$ with $0 < r <q$. Then  i) and 
Lemma \ref{lemma; binom} i)  show that 
$$\left|\frac{n_0!}{\varpi_p^{n_0}}/
\frac{m!}{\varpi_p^m}\right|
=\left|\binom{m}{r}^{-1} \frac{\varpi_p^r}{r!}\right|<1.$$

Finally, we show iii). Let $n$ be such that $n \equiv -1 \mod q$. 
We have  
$$\frac{(n+i)!}{ \varpi_p^{n+i}}/\frac{(n+q)!}{ \varpi_p^{n+q}}
=
\frac{(n+i)!}{(n+q)!}\varpi_p^{q-i}
=u  \frac{q}{\pi} \frac{(i-1)!}{\varpi_p^{i-1}} \frac{\pi \varpi_p^{q-1}}{ q!}
$$
where $u=\binom{n+q}{q-1}^{-1}\binom{n+i}{i-1}$ is a $p$-adic integer 
by Lemma \ref{lemma; binom} i). 
By Proposition \ref{proposition; factorial/varpi}, the $p$-adic 
(additive) valuation of 
the right hand side is positive. 
Since $v_p(\pi/\varpi_p)>0$, the $p$-adic (additive) valuation of 
$$\frac{(n+q)!}{ \varpi_p^{n+q}}/\frac{n!}{ \varpi_p^{n}}
=\binom{n+q}{q}\frac{q!}{ \pi \varpi_p^{q-1}}\frac{\pi}{\varpi_p}$$
is positive. 
\end{proof}

Next we investigate the absolute values of the coefficients of  
a power of the logarithm and the exponential map of 
the Lubin-Tate group. The case $k=1$ in the proposition 
 below is obtained in \cite{IS}. 

\begin{proposition}\label{lemma; lambda exp coefficients} 
We put $\partial={d}/{dt}$. Then we have 
$$
\left| \frac{\varpi_p^k \partial^n \lambda(t)^k}{k!n!} |_{t=0} \right| 
\leq  \underline{\rho}[k,n]^{-1}, 
\qquad \left| \partial^n \mathrm{exp}^k_{\mathcal{G}}(t) |_{t=0}\right |
 \leq | \varpi_p^n| \overline{\rho}[k,n].$$
\end{proposition}
\begin{proof}
The case for $n <k$ or $k=0$ is trivial. 
Suppose that $n \geq k \geq 1$. 
We first assume that the formal logarithm of $\mathcal{G}$ is given by 
$$\lambda(t)=\sum_{m=0}^\infty \frac{t^{q^m}}{\pi^m}.$$ 
Then it suffices to show   inequalities  
$$
\left|  \partial^n \lambda(t)^k |_{t=0} \right| \leq |k! \varpi_p^{n-k}|,  
\qquad \left| \partial^n \mathrm{exp}^k_{\mathcal{G}}(t) |_{t=0}\right |
 \leq |k! \varpi_p^{n-k}|.$$

When $k=1$, the inequality for $\lambda(t)$ is proven by direct calculations. 
We prove the general case by induction on $k$.  We have 
\begin{align*}
&\partial^n \lambda(t)^k |_{t=0} 
=k\partial^{n-1}( \lambda(t)^{k-1}\lambda'(t) )|_{t=0}\\
&=k\partial^{n-1} \sum_{m=0}^\infty  \lambda(t)^{k-1}\frac{q^mt^{q^m-1}}{\pi^m} |_{t=0}
=\sum_{m=0}^\infty  \binom{n-1}{q^m-1}\frac{q^m!k}{\pi^m}\partial^{n-q^m}\lambda(t)^{k-1} |_{t=0}.
\end{align*} 
Hence we have $|\partial^n \lambda(t)^k |_{t=0}| \leq  | k! \varpi_p^{n-k}|.$

We put $\mathrm{exp}^k_{\mathcal{G}}(t) =\sum_{n=k}^{\infty} a_n {t^n}$. 
We prove that $|n! a_n | \leq | k! \varpi_p^{n-k}|$ by induction for $n$. 
If $n=k$, this is true since $a_{k}=1$.  
We assume that the assertion is true for  integers less than $n$. 
Since $\mathrm{exp}^k_{\mathcal{G}}(\lambda(t)) =t^k$, we have 
$$t^k=a_k {\lambda(t)^k}+a_{k+1} {\lambda(t)^{k+1}}+ \cdots +a_n {\lambda(t)^n} + \cdots.$$
By i) and the inductive hypothesis, we have 
$$|a_{m} \partial^n {\lambda(t)^{m}}|_{t=0} | \leq | k! \varpi_p^{n-k}|$$
for $m<n$. Since $\partial^n {\lambda(t)^{n}}|_{t=0} =n!$ 
and $\partial^n {\lambda(t)^{m}}|_{t=0} =0$ for $n<m$, the assertion is also true for $n$. 
 
Now we consider a general parameter $s$. 
Then the logarithm  and the exponential for $\mathcal{G}$ with  parameter $s$ are 
of the form $\lambda(\phi(s))$  and 
$\psi(\exp_{\mathcal{G}}(s))$ for some 
$\phi(s), \psi(s) \in s\cO_{K}[[s]]^{\times}$. 
We put $\lambda(t)^k=\sum_{n=k}^\infty c_n^{(k)} t^n$ and 
$\lambda(\phi(s))^k=\sum d_n^{(k)} s^n$. 
Then  we have shown $|c_n^{(k)}| \leq |k! \varpi_p^{n-k}/n!|$. 
Since $d_n^{(k)} $ is a linear sum of $c_l^{(k)} $ $(k \leq l \leq n)$ with integral coefficients, 
we have 
$$\left|
\frac{\varpi_p^k d_n^{(k)} }{k!}
\right| \leq \max_{k\leq l \leq n} \left\{ \left |c_l^{(k)}\frac{\varpi_p^k}{k!} \right|\right\} \leq  \max_{k \leq l \leq n} \left\{ \left| 
\frac{\varpi_p^l}{l!}\right|\right\}= \underline{\rho}[k,n]^{-1}.  
$$
Hence we have the inequality  for the logarithm. The inequality  for the 
exponential is straightforward.  
\end{proof}


\begin{lemma}\label{lemma; coleman}
i) Suppose that $f(t) \in \cO_{K}[[t]]$ satisfies 
$f(t\oplus t_N)=f(t)$ for all $t_N \in \mathcal{G}[\pi^N]$. Then 
there exists 
a power series $g(t) \in \cO_{K}[[t]]$ such that  
$f(t)=g([\pi^N]t).$ \\
ii) There exists an integral power series $g_k(t) \in \cO_{K}[[t]]$ such that 
$$\pi^{-N}\sum_{t_N \in \mathcal{G}[\pi^N] } (t \oplus t_N)^k = g_k([\pi^N]t).$$
\end{lemma}
\begin{proof}
See  \cite{Col}, Chapter III. 
\end{proof}

We put 
$$F(t,X)=\prod_{t_N \in \mathcal{G}[\pi^N]}(1-(t \oplus t_N)X)=
1+\alpha_1(t)X +\cdots + \alpha_{q^N}(t)X^{q^N}.$$ 
For $\partial_X=\partial/\partial X$, we consider the power series 
\begin{equation}\label{equation; generating power sum}
\frac{\pi^{-N}\partial_XF(t,X)}{F(t, X)}=-\sum_{k=0}^\infty 
\left(\pi^{-N}\sum_{t_N \in \mathcal{G}[\pi^N]} (t \oplus t_N)^{k+1} \right)  X^k.
\end{equation}
By Lemma \ref{lemma; coleman} and the above formula, 
we have   $\pi^{-N}\partial_XF(t,X) \in 
\cO_{K}[[t]][X]$. 

\begin{proposition}\label{proposition; power estimate}
Let $k, n$ be non-negative integers and $N$ a natural number. Then 
we have 
\begin{equation}\label{equation; mainest 1}
\left |\pi^{-N}\sum_{t_N \in \mathcal{G}[\pi^N]} \partial^n_{\mathcal{G}}(t \oplus t_N)^{k} |_{t=0}
\right | \leq \left | {\pi^{Nn+k_0(1-\frac{1}{q-1})}\varpi_p^{n}}
 \right|
 \overline{\rho}\left(\left[\frac{k}{q^N}\right]\right)
\overline{\rho}(0)
\end{equation}
where $k_0=\max\{[k/q^N]-n, 0\}$.
We also have 
\begin{equation}\label{equation; mainest 2}
\left |\pi^{-N}\sum_{t_N \in \mathcal{G}[\pi^N]} \partial^n_{\mathcal{G}}(t \oplus t_N)^{k} |_{t=0}
\right | \leq \left | {\pi^{Nn}\varpi_p^{n}}
 \right|
 \overline{\rho}\left[0,n\right].
\end{equation}
Moreover, if $e \leq p-1$, we have 
\begin{equation}\label{equation; mainest 3}
\left|\pi^{-N}\sum_{t_N \in \mathcal{G}[\pi^N]} \partial^n_{\mathcal{G}} (t \oplus t_N)^{k} |_{t=0}
\right| \leq \left| \pi^{Nn}\varpi_p^{n}
\right|\overline{\rho}\left(\left[\frac{k}{q^N}\right]\right).
\end{equation}
\end{proposition}

\begin{proof}
We put $G(t,X)=F(0, X)-F(t, X)$, then $G(0,X)=G(t,0)=0$. We have 
$$\frac{1}{F(t, X)}
=\frac{1}{F(0,X)-G(t,X)}
=\sum_{l=0}^\infty \frac{G(t,X)^l}{F(0, X)^{l+1}} \quad \in \cO_K[[t,X]].  
$$
Since $G(0, X)=0$ and $G(t,X)$ is invariant for the translation $t \mapsto t_N$, 
it is of the form 
\begin{equation}\label{equation; GH}
G(t,X)=([\pi^N]t) H([\pi^N]t, X)
\end{equation}
 for some 
element $H$ in  $\cO_{K}[[t]][X]$.  Since $F(0, X) \equiv 1 \mod \pi$, 
the power series ${F(0,X)^{-l-1}}$ is equal to 
$$\sum_{m=0}^\infty \binom{-l-1}{m}(F(0,X)-1)^m
=\sum_{m=0}^\infty \binom{l+m}{m} \pi^m\left(\frac{1-F(0,X)}{\pi}\right)^m.$$ 
Hence we have 
\begin{align}
&\frac{\pi^{-N}\partial_X  F(t,X)}{F(t, X)}=\sum_{l=0}^\infty
 {\pi^{-N}\partial_XF(t,X)} \cdot
{G(t,X)^l}\cdot{F(0, X)^{-l-1}}\\
&=\sum_{l=0}^\infty  
\sum_{m=0}^\infty \binom{l+m}{m} \pi^m ({\pi^{-N}\partial_XF(t,X)}) 
{G(t,X)^l} \left(\frac{1-F(0,X)}{\pi}\right)^m. \label{equation; generating power sum 1}
\end{align}
To show the assertion for $k+1$, 
we look the coefficient of $X^k$ of (\ref{equation; generating power sum 1}).  
We consider the coefficients of the terms $X^a$, $X^b$ and $X^c$ with $a+b+c=k$ of 
$\pi^{-N} \partial_X F(t,X)$, $G(t,X)^l$ and $(1-F(0,X))^m\pi^{-m}$ respectively. 
Since $\mathrm{deg}\,\partial_X F(t,X)=q^N-1$, 
$\mathrm{deg}\,G(t,X)=q^N$ and $\mathrm{deg}\,(1-F(0,X))=q^N-1$ as  polynomials for $X$, 
we have $a \leq q^N-1,  b \leq l q^N$ and $c \leq m(q^N-1)$. 
Then  by (\ref{equation; GH}) the product of these coefficients is 
an integral linear combination of the terms of the form  
$$ \binom{l+m}{m} \pi^m G_l([\pi^N]t)$$
where $G_l(t)$ is a power series in $t^l \cO_{K}[[t]]$   and 
$l$, $m$ satisfies  
\begin{equation}\label{abc}
a+lq^N+m(q^N-1) \geq a+b+c = k.
\end{equation}
We estimate the absolute value of 
\begin{equation}\label{g_l}
 \binom{l+m}{m} \pi^m \partial_{\mathcal{G}}^n \,G_l([\pi^N]t)\;|_{t=0}. 
\end{equation}
By Proposition \ref{lemma; lambda exp coefficients}, we have 
$$\left| \partial_{\mathcal{G}}^n ([\pi^N]t)^d|_{t=0} \right|=
\left| \pi^{Nn}\frac{d^n}{dz^n} \exp^d_\mathcal{G}(z)|_{z=0} \right| \leq \left|  \pi^{Nn}
\varpi_p^n \right| \overline{\rho}[d,n].$$
Therefore, we have 
$$\left| \partial_{\mathcal{G}}^n G_l([\pi^N]t)|_{t=0} \right| 
\leq \left| \pi^{Nn}\varpi_p^n \right|  \overline{\rho}[l,n]  .$$ 
Hence we have (\ref{equation; mainest 2}). 
If $n<l$, then (\ref{g_l}) is zero  and there is nothing to prove. 
We assume that $n \geq l$. 
We let $l' \geq l$ be such that  $\overline{\rho}(l)= | l'!/\varpi_p^{l'}|$. 
Then 
\begin{align}
\left | \binom{l+m}{m} \pi^m \partial_{\mathcal{G}}^n 
 G_l([\pi]^Nt)|_{t=0} \right | \leq \left | \binom{l+m}{m} \pi^{m+Nn}
 \varpi_p^n \right | \overline{\rho}[l,n]   
\label{equation; power estimate 1}\\
\leq \left| \pi^{Nn}\varpi_p^n \frac{(l+m)!}{\varpi_p^{l+m}}\frac{(l'-l)!}{\varpi_p^{l'-l}} 
\binom{l'}{l}\frac{\varpi_p^m \pi^{m}}{m!} \right|. \label{equation; power estimate 2}
\end{align}

First we consider the case  $a \leq q^N-2$ or $m \not=0$. Then by (\ref{abc})
we have $$ l+m \geq \left[\frac{k+1}{q^N}\right].$$   
In particular, $ m \geq [(k+1)/q^N]-n$ and 
the value (\ref{equation; power estimate 2}) is less than or equal  to  
$$|{\pi^{Nn+k_0(1-\frac{1}{q-1})}\varpi_p^{n}}|
\overline{\rho}\left(\left[\frac{k+1}{q^N}\right]\right) \overline{\rho}(0)$$
where $k_0=\max\{[(k+1)/q^N]-n, 0\}$.
Hence in this case  we have (\ref{equation; mainest 1}).  
Suppose that $e \leq p-1$. If $l' < l+m$, then 
$\left| \varpi_p^m\right| < \left| \varpi_p^{l'-l}
\right | $ and hence 
the value  (\ref{equation; power estimate 2}) is less than 
$|\pi^{Nn}\varpi_p^{n}|\overline{\rho}\left(\left[\frac{k+1}{q^N}\right]\right).$ 
If $l' \geq l+m$, then 
$$
	  \overline{\rho}(l) =\left| \frac{l'!}{\varpi_p^{l'}} \right| \leq 
\overline{\rho}(l+m)  
\leq \overline{\rho}\left(\left[\frac{k+1}{q^N}\right]\right).$$ Hence 
the value  (\ref{equation; power estimate 1}) is also less than or equal to 
$|\pi^{m+Nn}\varpi_p^{n}|\overline{\rho}\left(\left[\frac{k+1}{q^N}\right]\right)$. 
Hence in this case  we have (\ref{equation; mainest 3}).

Finally we consider the case when  $a =q^N-1$ and $m=0$. 
Then 
the coefficient of $\pi^{-N} \partial_X F(t,X)$ of degree $a$ is 
$(q/\pi)^N\alpha_{q^N}(t)$, which is divisible by $[\pi^N]t$. Hence in this case 
the product of  the coefficient of $X^a$ in $\pi^{-N} \partial_X F(t,X)$, 
the coefficient of 
$X^b$ in $G(t,X)^l$
and the coefficient of $X^c$ in $(1-F(0,X))^m\pi^{-m}$ is 
an integral linear combination of terms in the form  
$G_{l+1}([\pi^N]t)$ for some $G_{l+1}(t) \in t^{l+1} \cO_{K}[[t]]$. In this case 
  $l$ satisfies  $l+1 \geq \left[{(k+1)}/{q^N}\right]$. Therefore 
\begin{align*}
\left| \partial_{\mathcal{G}}^n G_{l+1}([\pi]^Nt)|_{t=0} 
\right| \leq \left| \pi^{Nn}\varpi_p^n \right| \overline{\rho}[l+1,n]
 \leq \left| \pi^{Nn}\varpi_p^n 
 \right| \overline{\rho}\left(\left[\frac{k+1}{q^N}\right]\right).   
\end{align*}
If $n <l+1$, then (\ref{g_l}) is zero and there is nothing to prove.  We assume that 
$n \geq l+1$. 
In particular, by (\ref{abc}) we have $n \geq [(k+1)/q^N]$,  
and hence $k_0=\max\{[(k+1)/q^N]-n, 0\}=0$. 
Therefore we have (\ref{equation; mainest 1}) and (\ref{equation; mainest 3}). 
\end{proof}

\section{Integral structures on $p$-adic Fourier theory}

In this section, we give an explicit construction of  Schneider-Teitelbaum's  $p$-adic distribution associated to 
a rigid analytic function on the open unit disc. 

Let $\varphi(t)$ be a rigid analytic function on the open unit disc. 
We will construct a distribution $\mu_\varphi$ on $\cO_K$ such that
$$\int_{\cO_K} \, \exp(x \varpi_p\lambda(t)) d\mu_{\varphi}=\varphi(t). $$
If  we first had the Mahler like expansion for $K$-analytic functions, then 
it is easy to define the integral like as (\ref{equation: amice}), but as in  \cite{ST}, 
we first define the integral and then the Mahler like expansion for $K$-analytic function is 
shown by using this integral. 

We fix a Lubin-Tate formal group $\mathcal{G}$ with a parameter $\pi$ and 
denote its addition by $\oplus$. 
For $a \in \cO_K$ and a natural number $N$, we let 
\begin{equation}
\int_{a+\pi^N\cO_K} \; (x-a)^n \;d\mu_{\varphi}:=
\frac{1}{q^N\varpi_p^n}
\left(\partial_{\mathcal{G}}^n 
\sum_{t_N \in \mathcal{G}[\pi^N]}\varphi_a(t\oplus t_N)\right)\; \biggr |_{t=0}
\end{equation}
where 
$$
\varphi_a(t):=\exp(-a \varpi_p\lambda(t)) \varphi(t).
$$

We put $\varphi(t)=\sum_{k=0}^\infty c_k t^k$ and 
$\varphi_a(t)=\sum_{k=0}^\infty c^{(a)}_k t^k$. Then  
by Proposition \ref{proposition; power estimate}, we have 

\begin{align}
\left|\int_{a+\pi^N\cO_K} \; (x-a)^n \;d\mu_{\varphi} \right| & \leq 
\overline{\rho}(0)
 \left| \frac{\pi}{q}\right|^N  \left| \pi \right|^{Nn} 
 \sup_{k } \{ \,|c^{(a)}_k|  \overline{\rho}\left(\left[\frac{k}{q^N}\right]\right)  \,\}
\label{equation; fundamental integral estimate }  \\
&  \leq 
\overline{\rho}(0)
 \left| \frac{\pi}{q}\right|^N  \left| \pi \right|^{Nn} 
 \sup_{k } \{ \,|c_k| 
\overline{\rho}\left(\left[\frac{k}{q^N}\right]\right)  \,\}. 
\label{equation; fundamental integral estimate 2}
\end{align}
Here for the last estimate, we used the facts that 
$c^{(a)}_k$ is a integral liner combination of $c_0, \dots, c_k$ and 
the function $\overline{\rho}(m)$  for $m$ is decreasing. 

We define $\mu_\varphi$ on $LA_N(\cO_K, \bbC_p)$ as follows.
For an element $f$ of  $LA_N(\cO_K, \bbC_p)$,
suppose $f$ is of the form 
$\sum_{n=0}^\infty a_n (x-a)^n$ such that $a_n\pi^{nN} \rightarrow 0$ if $n \rightarrow \infty$
on $a+\pi^N\cO_K$. 
Then we define the integral of $f$ on $a+\pi^N\cO_K$ by 
\begin{equation}\label{equation; STdef}
\int_{a+\pi^N\cO_K} \; f(x) \;d\mu_{\varphi}:=
\sum_{n=0}^\infty\;a_n\int_{a+\pi^N\cO_K} \; (x-a)^n \;d\mu_{\varphi}. 
\end{equation}
We define 
\begin{equation}\label{def integral}
\int_{\cO_K} \; f(x) \; d\mu_{\varphi}=\sum_{a \!\!\!\mod \pi^N}
\int_{a+\pi^N\cO_K} f(x) \; d\mu_{\varphi}.
\end{equation}
We have to show the well-definedness of the integral. 

\begin{proposition}\label{proposition; estimate}
i) The integral (\ref{equation; STdef}) converges and 
does not depend on the choice of  the representative of $a \mod \pi^N$. 
The  integral (\ref{def integral}) does not depend on the choice of $N$. 
Hence $\mu_\varphi$ gives a well-defined element of $D(\cO_K,\bbC_p)$.\\
ii) For a polynomial $f(x)$, we have 
$$\int_{\cO_K} \; f(x) \;d\mu_{\varphi} 
=f( \varpi_p^{-1}\partial_{\mathcal{G}}) \varphi(t) |_{t=0}. $$
\end{proposition}
\begin{proof}
Since  $\overline{\rho}([k/q^N]) \leq C  k p^{-\frac{k}{eqN(q-1)}}$ 
for some constant $C$ which  depends only on $e,q$ and $N$, 
the value 
$\sup_{k } \{ \,|c_k| 
 \overline{\rho}\left(\left[\frac{k}{q^N}\right]\right)  \,\}$ is finite. 
Hence the convergence  follows from 
(\ref{equation; fundamental integral estimate 2}). 
We show that the integral  (\ref{equation; STdef}) 
depends 
only on the class of $a$ modulo $\pi^N$. 
Since the integral is convergent,  we may assume that  
$f$ is a monomial $(x-a)^n$. 
For $a'$ such that $a' \equiv a \mod \pi^N$, we put $b=a'-a$.  
Since  $$(x-a)^n |_{a'+\pi^N \cO_K}
=\sum_{l=0}^n  \binom{n}{l} b^{n-l}   (x-a')^l |_{a'+\pi^N\cO_K},$$
it suffices to show that 
\begin{equation} \label{equation; well-defined}
\int_{a+\pi^N\cO_K} \; (x-a)^n \;d\mu_{\varphi}
=\sum_{l=0}^n  \binom{n}{l} b^{n-l} \int_{a'+\pi^N\cO_K} \; (x-a')^l \;d\mu_{\varphi}.
\end{equation}
However,  we have 
\begin{align*}
\varpi_p^{-n} \partial_{\mathcal{G}}^n \varphi_a(t \oplus t_m)
&=\varpi_p^{-n}\partial_{\mathcal{G}}^n \left( \exp(b \varpi_p \lambda(t))
 \varphi_{a'}(t \oplus t_N)\right)\\
&= \exp(b \varpi_p \lambda(t)) \sum_{l=0}^n \binom{n}{l}b^{n-l}
\varpi_p^{-l}\partial_{\mathcal{G}}^l \left(\varphi_{a'}(t \oplus t_m)\right).  
\end{align*}
Hence (\ref{equation; well-defined}) follows. 

Now we  show that the integral (\ref{def integral}) 
does not depend on $N$.
It is sufficient to show the distribution relation 
\begin{equation}
\int_{a+\pi^N\cO_K} \; f(x) \; d\mu_{\varphi}=\sum_{b \equiv a  \;\mathrm{mod} \pi^N}
\int_{b+\pi^{N+1}\cO_K} \; f(x) \; d\mu_{\varphi}
\end{equation}
where the sum runs over a representative $b$ of $\cO_K/\pi^{N+1}$ such that 
$b \equiv a \mod \pi^N$. 
To show this, replacing $\varphi$ by $\varphi_a$, 
we may assume that $a=0$ and $f(x)=x^n$. 
Then 
\begin{align*}
q^{N+1}\varpi_p^{n} &\sum_{b \equiv 0 \!\!\!\mod \pi^N} 
\int_{b+\pi^{N+1}\cO_K} \; x^n \;d\mu_{\varphi} \\
&=
\sum_{b \equiv 0 \!\!\!\mod \pi^N}\sum_{i=0}^k \binom{n}{k} b^{n-k}
\left( \varpi_p^{n-k}\partial_{\mathcal{G}}^k
\sum_{t_{N+1} \in \mathcal{G}[\pi^{N+1}]}\varphi_b(t\oplus t_{N+1})\right) \biggr |_{t=0}\\
&=
\sum_{b \equiv 0 \!\!\!\mod \pi^N}\left(\partial_{\mathcal{G}}^n
\sum_{t_{N+1} \in \mathcal{G}[\pi^{N+1}]} 
\exp(b\varpi_p \lambda(t))\varphi_b(t\oplus t_{N+1})\right)\biggr |_{t=0}\\
&=
\sum_{t_{N+1} \in \mathcal{G}[\pi^{N+1}]}\left(\sum_{b \equiv 0 \!\!\!\mod \pi^N}
\exp(-b\varpi_p \lambda(t))|_{t=t_{N+1}}\right)\partial_{\mathcal{G}}^n
\varphi(t\oplus t_{N+1})\biggr |_{t=0}\\
&=q
\left(\partial_{\mathcal{G}}^n
\sum_{t_{N} \in \mathcal{G}[\pi^{N}]} 
\varphi(t\oplus t_{N})\right)|_{t=0}=q^{N+1}\varpi^{n}_p 
\int_{\pi^N\cO_K} \; x^n \; d\mu_{\varphi}.
\end{align*}
The above calculation is also true when $a=N=0$, and hence we have 
$$\varpi^{n}_p \sum_{b \in \cO_K/\pi} 
\int_{b+\pi\cO_K} \; x^n \;d\mu_{\varphi} 
=\partial_{\mathcal{G}}^n
\varphi(t) |_{t=0}.$$
From this  the assertion ii) follows. 
\end{proof}

For $\varphi(t)=\sum_{k=0}^\infty c_k t^k \in R^{\mathrm{rig}}$, 
we define  $\Vert \varphi\Vert_N$ by 
\begin{equation}
\Vert \varphi\Vert_N:=\max_{k}
\left\{\; |c_k| \overline{\rho}\left(\left[\frac{k}{q^N}\right] \right) \; \right\}.
\end{equation}
Since $\overline{\rho}\left(\left[\frac{k}{q^N}\right] \right)  \sim p^{-kr}$ where 
$r=1/eq^N(q-1)$, the value 
$\Vert \varphi\Vert_N$ is approximately, 
$$ \Vert \varphi\Vert_{\overline{\bold{B}}(p^{-r})}=\max_{x \in \overline{\bold{B}}(p^{-r})}\{\;|\varphi(x)|\;\}$$ 
where $\overline{\bold{B}}(p^{-r}) \subset \bbC_p$ is the closed disc with radius $p^{-r}$ at origin.  

\begin{lemma} 
For an element $a \in \cO_K$, we put $\varphi_a(t)=\exp(-a\varpi_p \lambda(t)) \varphi(t)$.  
as before. Then $\Vert \varphi_a\Vert_N=\Vert \varphi\Vert_N$.
\end{lemma}
\begin{proof}
It suffices to show $\Vert \varphi_a\Vert_N \leq \Vert \varphi\Vert_N$. 
This follows from the same argument  
showing  (\ref{equation; fundamental integral estimate 2}).
\end{proof}

Then Proposition \ref{proposition; power estimate} may
rewritten as follows, which is a precise version of Theorem 1.1 of the introduction.

\begin{theorem}\label{theorem; main}
i) 
Suppose that for $a \in \cO_K$, 
the function $f \in LA_N(\cO_K, \bbC_p)$ is given by a polynomial of degree  
$d$ on $a+\pi^N \cO_K$. For  $\varphi_k(t)=t^k$,  we have 
\begin{align}\label{equation; STint-estimate}
\left|\int_{a+\pi^N\cO_K} \; f(x) \;d\mu_{\varphi_k} \right| & \leq 
\overline{\rho}[0,d]
 \left| \frac{\pi}{q}\right|^N \;\Vert f \Vert_{a,N}.  
\end{align} 
We also have 
\begin{align}\label{equation; STint-estimate3}
\left|\int_{a+\pi^N\cO_K} \; f(x) \;d\mu_{\varphi_k} \right| \; \leq \; \overline{\rho}(0) 
\left| \frac{\pi^{k_0(1-\frac{1}{q-1})+N}}{q^N}\right|  \;\Vert f \Vert_{a,N} \overline{\rho}\left(\left[\frac{k}{q^N}\right] \right) 
\end{align}
where $k_0=\max\{[k/q^N]-d, 0\}$. 
Moreover,  if $e \leq p-1$, then we have 
\begin{align}\label{equation; STint-estimate4}
\left|\int_{a+\pi^N\cO_K} \; f(x) \;d\mu_{\varphi_k} \right| \; \leq \;  
\left| \frac{\pi}{q}\right|^N  \;\Vert f \Vert_{a,N} \overline{\rho}\left(\left[\frac{k}{q^N}\right] \right) .  
\end{align}
ii) We have 
\begin{align}\label{equation; STint-estimate1}
\left|\int_{a+\pi^N\cO_K} \; f(x) \;d\mu_{\varphi} \right| \; \leq \; 
 \overline{\rho}(0) 
 \left| \frac{\pi}{q}\right|^{N}  \;\Vert f \Vert_{a,N} 
 \Vert \varphi\Vert_{N}.  
\end{align}
Moreover, if $e \leq p-1$, then 
\begin{align}\label{equation; STint-estimate 2}
\left|\int_{a+\pi^N\cO_K} \; f(x) \;d\mu_{\varphi} \right| \; \leq \;
 \left| \frac{\pi}{q}\right|^{N}  \;\Vert f \Vert_{a,N} 
 \Vert \varphi\Vert_{N}. 
\end{align}
\end{theorem}

\begin{corollary}\label{corollary; main}
We have 
\begin{align}\label{equation; STint-estimate cor 1}
\left|\int_{a+\pi^N\cO_K} \; f(x) \;d\mu_{\varphi} \right| \; \leq \; 
p^{\frac{p}{p-1}+\frac{1}{e(q-1)}}
 \overline{\rho}(0) 
 |\pi|^N  \;\Vert f \Vert_{a,N} 
 \Vert \varphi\Vert_{\bold{B}'(p^{-r})}  
\end{align}
where $r=1/eq^N(q-1)$ and 
\begin{equation}
\Vert \varphi\Vert_{\bold{B}'(p^{-r})}:=\max_{k}
\left\{\;  |c_k| k p^{-kr}\; \right\}.
\end{equation}
Moreover, if $e \leq p-1$, then 
\begin{align}\label{equation; STint-estimate cor 2}
\left|\int_{a+\pi^N\cO_K} \; f(x) \;d\mu_{\varphi} \right| \; \leq \; 
p^{\frac{p}{p-1}+\frac{1}{e(q-1)}}
 |\pi|^N  \;\Vert f \Vert_{a,N} 
 \Vert \varphi\Vert_{\bold{B}'(p^{-r})}.  
\end{align}
\end{corollary}
\begin{proof}
The formula follows from 
$$\overline{\rho}\left(\left[\frac{k}{q^N}\right] \right)  \leq 
k q^{-N} p^{\frac{p}{p-1}+\frac{1}{e(q-1)}-\frac{k}{eq^N(q-1)}}.$$
\end{proof}

As before, we define polynomials $P_{n}$ by 
$$\exp(x \lambda(T))=\sum_{n=0}^\infty \; P_n(x) T^n.$$
Then by formal computation,  we have 
$$P_k(\partial_{\mathcal{G}}) \varphi(t)|_{t=0}=\frac{1}{k!}\partial^k \varphi(t) |_{t=0}$$ 
where $\partial=d/dt$ (for example, formula $6$ of Lemma 4.2 of \cite{ST}). 
We let  $\varphi_n(t)=t^n$ and $\mu_{\varphi_n}$ the 
distribution  associated to $\varphi_n(t)$. Then by 
Proposition \ref{proposition; estimate} ii) 
we have 
$$\int_{\cO_K} \; P_k(x \varpi_p) \;d\mu_{\varphi_n}=\sum_{n=0}^\infty\; P_k( \partial_\mathcal{G})
\varphi_n(t)|_{t=0}=
\begin{cases}
1  \quad(k=n)\\
0  \quad (k \not=n). 
\end{cases}$$
Hence if $\varphi(t)=\sum_{k=0}^\infty c_k t^k$, then 
$$\int_{\cO_K} \; P_k(x \varpi_p) \;d\mu_\varphi=c_k.$$
Equivalently,    
$$\varphi(t)=\int_{\cO_K} \, \exp(x \varpi_p\lambda(t)) d\mu_\varphi.$$

\begin{proposition}\label{proposition; p_n}
For $N \geq 1$, we have 
 $$\left|\frac{q}{\pi}\right|^N 
\overline{\rho}\left(\left[\frac{n}{q^N}\right]\right)^{-1}
c^{-1} \leq 
|| P_n(x \varpi_p)||_{N} \leq \underline{\rho}\left(\left[\frac{n}{q^N}\right]\right)^{-1}$$
where $c=1$ if $e \leq p-1$ and $c=\overline{\rho}(0)$, otherwise. 
\end{proposition}
\begin{proof}
We have 
\begin{align*}
1&=\left|\int_{\cO_K} P_n(x \varpi_p) \;d\mu_{\varphi_n} \right | 
\leq \max_{a}\{ \left|\int_{a+\pi^N\cO_K}  P_n(x \varpi_p) \;d\mu_{\varphi_n} \right | \}\\
& \leq \left|\frac{\pi}{q}\right|^N || P_n(x \varpi_p)\Vert_{N} \,\,
\overline{\rho}\left(\left[\frac{n}{q^N}\right]\right)
\overline{\rho}(0).
\end{align*}
Similarly, if $e \leq p-1$, then by using 
(\ref{equation; STint-estimate4}), we get the lower estimate. 

We show the upper estimate.  
We put $P_n( x \pi^N  \varpi_p)=\sum_{k=1}^n a^{(n)}_k x^k$ for $n \geq 1$. 
By definition of $P_n$, 
the value 
$
 a^{(n)}_k$ is the coefficient of $t^n$ of 
${\varpi_p^k \lambda( [\pi^N] t)^k}/{k!}
$. 
Since $\underline{\rho}(k)$ is decreasing with
$k$, we may assume that $\lambda(t)=\sum_{l=0}^\infty t^{q^l}/\pi^l$. 
Since $[\pi^N]t \equiv t^{q^N} \!\!\mod \pi$,  we have 
$$\lambda([\pi^N]t) \equiv \lambda(t^{q^N}) +\pi t f(t)$$
for some $f(t) \in \cO_{\bbC_p}[[t]]$. (cf. \cite[Lemma 4]{Ho}.)
Hence we have 
$$ \frac{\varpi_p^k\lambda([\pi^N]t)^k}{k!} =
\sum_{i=0}^k t^i f(t)^{i}  \frac{\varpi_p^i \pi^{i}}{i!} 
\frac{\varpi_p^{k-i}\lambda(t^{q^N})^{k-i}}{(k-i)!}.  
$$
Therefore by Proposition \ref{lemma; lambda exp coefficients} we have 
$$|a^{(n)}_k| \leq 
\underline{\rho}\left(\left[\frac{n}{q^N}\right]\right)^{-1}.$$
Hence we have $
|| P_n(x \varpi_p) ||_{0, N} 
 \leq \underline{\rho}\left(\left[\frac{n}{q^N}\right]\right)^{-1}.
$
Then by the formula before Lemma 4.4 of \cite{ST}, for  $a \in \cO_K$, we have 
$$
|| P_n(x \varpi_p) ||_{a, N} \leq \max_{0 \leq i \leq n} {|| P_i(x \varpi_p) ||_{0, N}}
 \leq \underline{\rho}\left(\left[\frac{n}{q^N}\right]\right)^{-1}.
$$
\end{proof}

Now we prove that our definition of the distribution  coincides 
with that of Schneider-Teitelbaum. 
Namely, it has the characterization property 
(\ref{equation; structure of distribution K}). 

\begin{theorem}
Let $\mu_\varphi$ be the distribution associated to a rigid analytic function 
$\varphi(t)$ on the open unit disc. 
Then 
$$\varphi(t)=\int_{\cO_K} \, \exp(x \varpi_p\lambda(t)) d\mu_\varphi.$$
Conversely, for every distribution $\mu$, there exists a unique 
rigid analytic function  $\varphi$  such that  
$\mu=\mu_\varphi$.   Then $\varphi$ is the Fourier transform of $\mu$, and we have $F_{\mu_\varphi} = \varphi$.
In particular, we have an isomorphism of algebras, 
\begin{equation*}
 D(\cO_{K}, \bbC_p)\cong R^{\mathrm{rig}}.
\end{equation*}
\end{theorem}
\begin{proof}
We have already shown the first assertion. 
For a given $\mu$, we put 
$$c_k:=\int_{\cO_K} \; P_k(x \varpi_p) \;d\mu.$$
Since the  distribution is a continuous linear  operator on the 
Banach space $LA_N(\cO_K, \bbC_p)$ for every natural number $N$,  
there exists a positive constant $C$ depending only on $\mu$ and $N$  
such that 
$$|c_k| = \left| \int_{\cO_K} \; P_k(x \varpi_p) \;d\mu\right| \leq C
 \Vert P_k(x \varpi_p)\Vert_{N} \leq C p^{-\frac{1}{p-1}+\frac{k}{eq^N(q-1)}}$$
where for the last inequality, we used Proposition \ref{prop, gamma estimate} and 
Proposition \ref{proposition; p_n}. 
Hence for any $0 \leq r <1$, if we choose sufficiently large $N$, 
we have $|c_k| r^k \rightarrow 0$ when $k \rightarrow \infty$. 
Hence $\varphi(t)=\sum_{k=0}^\infty c_k t^k$ is a rigid analytic function 
on the open unit disc. Then by construction 
$$\varphi(t)=\int_{\cO_K}  \exp(x \varpi_p\lambda(t)) d\mu.$$
Since the function  $(x-a)|_{a+\pi^N\cO_K}$ is  given by 
$$\frac{1}{q^{N}\varpi_p^n}\partial^n_{\mathcal{G}} 
 \left(\sum_{t_N \in \mathcal{G}[\pi^N]} \exp((x-a) \varpi_p\lambda(t))|_{t=t \oplus t_N}\right)|_{t=0},$$
we have  
$$\int_{a+\pi^N\cO_K} \,(x-a)^n d\mu=\frac{1}{q^N\varpi_p^n} 
\partial^n_{\mathcal{G}} \sum_{t_N  \in \mathcal{G}[\pi^N]}\, 
\varphi_a(t\oplus t_N)|_{t=0}=\int_{a+\pi^N\cO_K} \,(x-a)^n d\mu_\varphi.$$
Since $\pi^{-nN}(x-a)^n|_{a+\pi^N\cO_K}$ for $a \in \cO_K$ and $n=0,1, \cdots$ 
are topological generators of $LA_n(\cO_K, \bbC_p)$, 
we have 
$$\int_{\cO_K} f(x) d\mu=\int_{\cO_K} f(x) d\mu_\varphi$$
for all $f \in LA_N(\cO_K, \bbC_p)$. Hence $\mu=\mu_\varphi$. 
\end{proof}

Now we prove Theorem 1.2.

\begin{theorem}\label{theorem; mahler L_N}
i) The series $\sum_{n=0}^\infty a_n P_n(x \varpi_p)$ converges to 
an element of  $LA_N(\cO_{K}, {\bbC_p})_{0}$ for $a_n$  satisfying
\begin{align*}
	|a_n| & \leq \underline{\rho}\left(\left[ \frac{n}{q^N}\right]\right), &
	\lim_{n \rightarrow 0} &|a_n|/\underline{\rho}\left(\left[ \frac{n}{q^N}\right]\right)= 0.
\end{align*} 
ii) If  $f(x) \in LA_N(\cO_{K}, \mathbb{C}_p)_{0}$, then it has 
an expansion 
$$f(x)=\sum_{n=0}^\infty a_n P_n(x \varpi_p)$$ 
of the form 
\begin{align*}
	|a_n| & \leq c \left |\frac{\pi}{q} \right |^N  \overline{\rho}\left(\left[ \frac{n}{q^N}\right]\right), &
	\lim_{n \rightarrow 0} &|a_n|/\overline{\rho}\left(\left[ \frac{n}{q^N}\right]\right)= 0,
\end{align*}
where $ c=1$ if $e \leq p-1$,  and  $c =\overline{\rho}(0) $, otherwise. 
\end{theorem}
\begin{proof}
i) follows from Proposition \ref{proposition; p_n}. For ii), 
we proceed as in the proof of Theorem 4.7 of \cite{ST} except the estimate of 
the Mahler coefficients. 
We put  $$a_n:=\int_{\cO_K} f(x) \;d\mu_{\varphi_n}.$$
Then by Theorem \ref{theorem; main},  we have 
$$
|a_n|=\left|\int_{\cO_K} \; f(x) \;d\mu_{\varphi_n} \right| \leq c \left |\frac{\pi}{q} \right |^N
\overline{\rho}\left(\left[ \frac{n}{q^N}\right]\right). 
$$
We next prove the limit in ii).
We may assume that $f(x)=\sum_{i=0}^\infty c_i (x-a)^i$ on 
$a+\pi^N \cO_{K}$ and $f(x)=0$   outside of $a+\pi^N \cO_{K}$. 
For a given $\epsilon>0$, we can take $N_0$ so that 
$$\Vert \sum_{i=N_0}^\infty c_i (x-a)^i \Vert_{a,N}<\epsilon.$$
Hence by  (\ref{equation; STint-estimate3}), we have 
\begin{align}\label{equation; c_1}
\left|\int_{a+\pi^N\cO_K} \; \sum_{i=N_0}^\infty c_i (x-a)^i \;d\mu_{\varphi_n} \right| \; \leq \; 
\epsilon\; C_1  \overline{\rho}\left(\left[\frac{n}{q^N}\right] \right)  
\end{align}
where $C_1$ is a positive constant independent of $n$. 
On the other hand, also by (\ref{equation; STint-estimate3}),  we have 
\begin{align}\label{equation; c_2}
\left|\int_{a+\pi^N\cO_K} \; 
\sum_{i=0}^{N_0} c_i (x-a)^i
 \;d\mu_{\varphi_n} \right| \; \leq \;C_2 p^{-\frac{n_0}{e} \left(1-\frac{1}{q-1}\right)} 
  \overline{\rho}\left(\left[\frac{n}{q^N}\right] \right) 
\end{align}
where $n_0=\max\{[n/q^N]-N_0, 0\}$ and $C_2$ is a positive constant independent of $n$. 
Hence we have 
\begin{align}\label{equation; c_3}
\left|\int_{a+\pi^N\cO_K} \; f(x)\;d\mu_{\varphi_n} \right| \; \leq \; 
\epsilon\; C_1  \overline{\rho}\left(\left[\frac{n}{q^N}\right] \right) 
\end{align}
for sufficiently large $n$. 
Hence we have $|a_n|/\overline{\rho}\left(\left[ \frac{n}{q^N}\right]\right) \rightarrow 0$ when $n \rightarrow \infty$.
Then by i),  the series $\sum_{k=0}^\infty a_n P_n(x\varpi_p)$ converges to 
a function in $LA_N(\cO_K, \bbC_p)$. 
We put 
$$g(x)=f(x)-\sum_{k=0}^\infty a_n P_n(x\varpi_p).$$ Then we have 
$\int_{\cO_K}  g(x) d\mu_{\varphi_n}=0$ for all $n$, and 
hence $\int_{\cO_K}  g(x) d\mu=0$ for all distribution $\mu$. 
Considering the Dirac distribution $\delta_a: h \mapsto h(a)$, 
we have  $g(a)=0$ for any $a$. Hence 
$f(x)=\sum_{n=0}^\infty a_n P_n(x \varpi_p)$. 
\end{proof}

\begin{corollary}\label{corollary; basis}
Suppose
$$e_{N,n}=\underline{\gamma}\left(\left[\frac{n}{q^N}\right] \right) P_n(x \varpi_p), \qquad (n=0,1, \cdots),$$  
where $\underline{\gamma}(u)$ is an element in $\bbC_p$ satisfying $\underline{\rho}(u)=|\underline{\gamma}(u)|$.
If $L_N$ is the $\cO_{\bbC_p}$-module topologically generated by 
$e_{N,n}$, then 
$$\overline{\rho}(0)^{-2}\left |\frac{q}{\pi} \right |^N LA_N(\cO_K, {\bbC_p})_{0}\;\subset\; L_N
\; \subset \; LA_N(\cO_K, {\bbC_p})_{0}.$$
\end{corollary}

In particular, the functions $e_n$ form a topological basis of 
the Banach space  $LA_N(\cO_K, {\bbC_p})$. 
Moreover, if $e \leq p-1$, then 
$$\left |\frac{q}{\pi} \right |^{N+1} LA_N(\cO_K, {\bbC_p})_{0}\;\subset\; L_N
\; \subset \; LA_N(\cO_K, {\bbC_p})_{0}.$$
In addition, if $\cO_K=\bbZ_p$, we recover Amice's result, namely,  
$$\left[\frac{n}{p^N}\right]! \binom{x}{n}$$
 for $n=0,1, \cdots$ form 
a topological basis of  $LA_N(\bbZ_p, {\bbC_p})_{0}$.

\section{Relations to Katz's and Chellali's results.}

As an application, we reprove Katz's and Chellali's results (\cite{Ch}, \cite{Ka2}) 
by using the $p$-adic Fourier theory. 

First we recall results of Katz \cite{Ka2} and Chellali \cite{Ch}. 
Let $E$ be an elliptic curve with complex multiplication by 
the ring of integer $\cO_{\bsK}$ of an imaginary quadratic field $\bsK$. 
For simplicity, we assume that $E$ is defined over $\bsK$ 
 and fix a 
Weierstrass model
$$y^2=4x^3-g_2x-g_3, \qquad g_2, g_3 \in \cO_\bsK$$
of $E/\bsK$. 
Let $p$ be an odd prime. 
We assume that $p$  is inert  in $\bsK$ and does not divide 
the discriminant of the above Weierstrass model, or equivalently, 
$E$ has good {\it supersingular} reduction at $p$. 
Then 
the Bernoulli-Hurwitz number ${BH}(n)$ is defined by 
$$
\wp(z)=\frac{1}{z^2}+\sum_{n \geq 2} \; \frac{{BH(n+2)}}{n+2} 
\frac{z^n}{n!},
$$
where $\wp(z)$ is the Weierstrass $\wp$-function for the model. 
Let $\epsilon$ be a root of unity in $\cO_\bsK$ such that 
the multiplication by $-\epsilon p$ gives the Frobenius 
$(x, y) \mapsto (x^{p^2}, y^{p^2})$ of $E \mod p$. 
Let $\gamma$ be a unit in the Witt ring $W(\overline{\mathbb{F}}_{p})$ such that 
$$\gamma^{p^2-1}=-\epsilon^{-1} \frac{p^2!}{p^{p+1}(p^2-1)}.$$
For a fixed $b \in \cO_K$ prime to $p$, we put 
$$L(n)=\frac{(1-b^{n+2}) (1-p^n)}{\gamma^n p^{[np/(p^2-1)]}}\frac{BH(n+2)}{n+2}.$$

\begin{theorem}[Katz \cite{Ka2}] \label{theorem; Ka2} 
The number $L(n)$ is integral. 
Let $l$ and $n$ be non-negative integers. Then 
$$L(n+p^l(p^2-1)) \equiv L(n) \quad \mod p^{l}.$$
\end{theorem}

Later, Chellali \cite{Ch} refined the congruences as follows. 

\begin{theorem}[Chellali \cite{Ch}] \label{theorem; Che}
Let $l$ and $n$ be non-negative integers. If $n \not\equiv 0 \mod p^2-1$, we have 
$$L(n+p^l(p^2-1)) \equiv L(n) \quad \mod p^{l+1}.$$
If $n \equiv 0 \mod p^2-1$ and  $n \not=0$, 
put $L'(n)=L(n)/n$, then 
$$L'(n+p^l(p^2-1)) \equiv L'(n) \quad \mod p^{l+1}.$$
\end{theorem}

In the following, let $K$ be the unramified quadratic extension of $\bbQ_p$ and  
let $\mathcal{G}$ be 
 the Lubin-Tate group  of  height $h=2$ associated to the uniformizer $\pi=-\epsilon p$. 
We assume that $[\pi]T=\pi T+T^q$ for $q=p^2$ is an endomorphism of $\mathcal{G}$. 
It is known that the formal group of $E$ at $p$ is isomorphic to $\mathcal{G}$.

\begin{proposition}\label{proposition; moment congruence} 
Let $\varphi$ be an integral power series  and 
let $\mu_\varphi$ be the corresponding distribution associated to 
$\varphi$. \\
i) We have 
\begin{equation*}
\left| \int_{\cO^\times_K} \; x^n \; d\mu_\varphi \right|
 \leq p.
\end{equation*}
ii) If $m \equiv n \mod p^l(q-1)$, then 
\begin{equation*}
\left| \int_{\cO^\times_K} \; (x^m-x^n) \; d\mu_\varphi \right|
 \leq 
p^{-l+\frac{p}{q-1}}. 
\end{equation*}
iii) If $(q-1)|n$ and $m \equiv n \mod p^l(q-1)$, then 
\begin{equation*}
\left| \int_{\cO^\times_K} \; \left( \frac{x^m-1}{m}-\frac{x^n-1}{n} \right) \; d\mu_\varphi \right|
 \leq 
p^{-l-1+\frac{2 p}{q-1}}.  
\end{equation*}
\end{proposition}
\begin{proof} 
We have 
$$\int_{a+\pi \cO_K} \; x^n \; d\mu_\varphi =a^n \int_{a+\pi \cO_K} d\mu_\varphi
+ \sum_{k=1}^n \int_{a+\pi \cO_K} \;\binom{n}{k} (x-a)^k a^{n-k}\; d\mu_\varphi.$$
Then by the estimate (\ref{equation; STint-estimate})
the absolute value of the first integral is less than or equal to $p$.  
By the estimate (\ref{equation; STint-estimate4}), the absolute value of the second integral is  also 
less than or equal to $p$ since 
$||(x-a)||_{a,1} \overline{\rho}(0) = 1$. 
We put $m-n=k(q-1)$.
Then 
\begin{align*}
x^m-x^n&=x^n \sum_{i=1}^k \binom{k}{i} (x^{q-1}-1)^i \\
&=kx^n (x^{q-1}-1)+x^n \sum_{i=2}^k k \binom{k-1}{i-1} \frac{(x^{q-1}-1)^i}{i} \\
&=k\left(c_0+c_1(x-a)+c_2\frac{(x-a)^2}{2}+c_3\frac{(x-a)^3}{3}+\cdots \right)
\end{align*}
where $c_i$ are integers satisfying  $p|c_0$. 
Since $\Vert {(x-a)^i}/{i}\Vert_{a,1}
\leq p^{-2}$ for $i \geq 2$, the assertion ii)  follows from 
 the estimates (\ref{equation; STint-estimate}). 

For an integer $s$, we have 
\begin{align*}
\frac{(x^{q-1})^s-1}{s}&= \sum_{i=1}^\infty \frac{(\log_p x^{q-1})^i}{i!} s^{i-1}
 =\sum_{i=1}^\infty \sum_{n=i}^\infty c_{i,n}\frac{(x^{q-1}-1)^{n}}{n!}\\
&= \sum_{i=1}^\infty \sum_{j+k \geq i}^\infty c_{i,j,k} \frac{\pi^k}{k!}\frac{(x-a)^j}{j!} s^{i-1}
\end{align*}
for some integers $c_{i, n}$ and $c_{i,j,k}$. 
If we write $m=s_1(q-1)$ and $n=s_2(q-1)$, then 
\begin{align*}
& \frac{(x^{q-1})^{s_1}-1}{s_1}-
\frac{(x^{q-1})^{s_2}-1}{s_2}
=  \sum_{i \geq 2, j+k \geq i}^\infty 
c_{i,j,k} \frac{\pi^k}{k!}\frac{(x-a)^j}{j!} (s^{i-1}_1-s^{i-1}_2) 
\end{align*}
By the estimate (\ref{equation; STint-estimate}), 
the integral of  $\frac{\pi^k}{k!}\frac{(x-a)^j}{j!}$ is 
divisible by $p^{1-\frac{2p}{q-1}}$.  The assertion iii) follows from this fact. 
 \end{proof}

For $b \in \cO_K$ prime to $p$, we put 
$$\wp_b(z)=(1-b^2[b]^*) \wp(z) $$
and $\phi(t)=\wp_b(z)|_{z=\lambda(t)}$. 
Then $\wp_b(z)$ has no pole at $z=0$ and  
$$\wp_b(z)=\sum_{n \geq 2}\; (1-b^{n+2}) \, \frac{BH(n+2)}{n+2} \frac{z^n}{n!}$$
It is known that  $\phi(t)$ is an integral power series. 
Similarly, for $c \in \cO_K$ prime to $p$, we put 
$$\zeta_c(z)=(c-[c]^*) \zeta(z), \qquad 
\zeta_{b,c}(z)=(1-b[b]^*)\zeta_c(z)$$ 
where $\zeta(z)$ is the Weierstrass zeta function 
and $\psi(t)=\zeta_{b,c}(z)|_{z=\lambda(t)}$. 
Note that $\zeta_c(z)$ is double periodic and $\zeta_{b,c}(z)$ has 
no pole at $z=0$. 
Then 
$$\zeta_{b,c}(z)=\sum_{n \geq 3}\; (c-c^{n})(1-b^{n+1})\, \frac{BH(n+1)}{n+1} \frac{z^n}{n!}
$$
and $\psi (t)$ is an integral power series.  

\begin{lemma}
$$\sum_{z_0 \in \frac{1}{p}\Gamma/\Gamma}{\wp}_b(z+z_0)=p^2 \wp_b(pz), 
\sum_{z_0 \in \frac{1}{p}\Gamma/\Gamma}{\zeta}_c(z+z_0)=p \zeta_c(pz).$$
\end{lemma}
\begin{proof}
It is known that 
$$\sum_{z_0 \in \frac{1}{p}\Gamma/\Gamma}{\wp}(z+z_0)=p^2 \wp(pz).$$
The first formula follows from this. 
The above formula also show that for a set  $S$ of representatives of $\frac{1}{p}\Gamma/\Gamma$, 
there exists a constant $A(S)$ such that 
$$\sum_{z_0 \in S}{\zeta}(z+z_0)=p \zeta(pz) +A(S).$$
We take $S$ so that $S=-S$. Then since $\zeta(z)$ is an odd function, $A(S)$ should be zero. 
Therefore, 
$$\sum_{z_0 \in S}{\zeta}_c(z+z_0)=p \zeta_c(pz).$$ 
Since $\zeta_c(z)$ is an  elliptic function, the left hand side does not depend on 
the choice of $S$. 
\end{proof}

\begin{proposition}
We put $B(n)=BH(n+2)/(n+2)$ if $n \geq 2$ and $0$ if $n=-1, 0, 1$. 
For $n \geq 0$, we have 
$$\varpi_p^n \int_{\cO_K^\times} x^n d\mu_{\phi}=(1-p^{n})(1-b^{n+2})B(n),$$
$$\varpi_p^n \int_{\cO_K^\times} x^n d\mu_{\psi}=(1-p^{n-1})
(c-c^{n})(1-b^{n+1})B(n-1).$$
\end{proposition}
\begin{proof}
Since $\wp_b(z)$ and $\zeta_{b,c}(z)$ are double periodic, 
for $t_0 \in \mathcal{G}[p]$ 
we have $\psi(t \oplus t_0)={\zeta}_{b,c}(z +z_0) |_{z=\lambda(t)}$ and 
$\phi(t \oplus t_0)=\wp_{b}(z +z_0) |_{z=\lambda(t)}$
where $z_0$ is an image of $t_0$ by  
$\mathcal{G}[p] \rightarrow E[p] \rightarrow \frac{1}{p} \Gamma /\Gamma$. 
(See for example, \cite{BK1}, Lemma 2.18.)
From this fact and the previous lemma, we have 
$$\phi(t)-\frac{1}{q} \sum_{t_0 \in \mathcal{G}[p]} \phi(t \oplus t_0)
=(\wp_{b}(z)-\wp_{b}(pz))\;|_{z=\lambda(t)},
$$
$$\psi(t)-\frac{1}{q} \sum_{t_0 \in \mathcal{G}[p]} \psi(t \oplus t_0)
=(\zeta_{b,c}(z)-p^{-1}\zeta_{b,c}(pz))\;|_{z=\lambda(t)}. 
$$
Hence 
\begin{align*}
\varpi_p^n \int_{\cO_K^\times} x^n d\mu_{\phi}&=
\partial_{\mathcal{G}}^n \left(\phi(t)-\frac{1}{q} \sum_{t_0 \in \mathcal{G}[p]} \phi(t \oplus t_0)
 \right) \biggr |_{t=0}\\
&= \partial_z (\wp_{b}(z)-\wp_{b}(pz))|_{z=0}
=(1-p^{n})(1-b^{n+2})B(n). 
\end{align*}
The other equality is also shown similarly. 
\end{proof}

We put 
$$c(n)=(1-p^{n})(1-b^{n+2})\frac{BH(n+2)}{n+2}.$$

\begin{corollary}\label{corollary; c(n)}
i) We have 
$$\left |\frac{c(n)}{\varpi_p^{n}} \right | \leq p.$$
Furthermore, if $n \equiv 0 \mod q-1$, then 
$$\left | \frac{c(n)}{\varpi_p^n} \right | \leq p^{\frac{p}{q-1}}.$$ 
ii) Suppose that $m \equiv n \mod p^l(q-1)$. Then 
$$\frac{c(m)}{\varpi_p^m} \equiv \frac{c(n)}{\varpi_p^n} \quad \mod p^{l-\frac{p}{q-1}}\cO_{\bbC_p}.$$
Furthermore, if $n \not\equiv 0 \mod q-1$, then 
$$\frac{c(m)}{\varpi_p^m} \equiv \frac{c(n)}{\varpi_p^n} \quad \mod p^{l}\cO_{\bbC_p}.$$
If  $n \equiv 0 \mod q-1$, then 
$$\frac{c(m)}{m\varpi_p^m} \equiv \frac{c(n)}{n\varpi_p^n} \quad \mod p^{l+1-\frac{2p}{q-1}}.$$
\end{corollary}
\begin{proof}
For i), the first inequality follows from Proposition \ref{proposition; moment congruence} i) 
for $\mu_\phi$. The second inequality follows from 
Proposition \ref{proposition; moment congruence} ii) for $l=0$. 
Note that  $\int_{\cO_K^\times}  d\mu_{\phi}=0$. 
For ii), the first and third congruences follow from Proposition \ref{proposition; moment congruence} 
for  $\phi$, and 
the second inequality for $\psi$. 
\end{proof}

Next,  we compare $c(n)$ with $L(n)$. 

\begin{lemma}\label{lemma; ugamma}
We choose $u \in \bbC_p$ so that $\varpi_p^{q-1}=p^{p}u^{q-1}$. 
Then $u$ is a unit of $\cO_{\bbC_p}$ and 
$$ \left(\frac{u}{\gamma}\right)^{q-1} \equiv 1 \quad \mod p.$$
\end{lemma}
\begin{proof}
Simple calculation shows the valuation of $u$ is zero. 
We have 
$\lambda(t)=t+\theta t^q+\cdots $ with $\theta=1/\epsilon(p^q-p)$. 
The $q$-th coefficient of the integral power series 
$\exp(\varpi_p \lambda(t))$ is 
$$\frac{\varpi_p^q}{q!}+\varpi_p \theta =\varpi_p \theta \left( 
\frac{\varpi_p^{q-1}  }{\theta q!}+1\right).$$
Since $\varpi_p \theta$ is not integral, the valuation 
$v_p(({\varpi_p^{q-1}  }/{\theta q!})+1) \geq 1$. Thus 
$$\frac{\varpi_p^{q-1}  }{\theta q!}+1 \equiv 
\left(\frac{u}{\gamma}\right)^{q-1} \frac{(1-p^{q-1})}{(q-1)}+1
\equiv -\left(\frac{u}{\gamma}\right)^{q-1}  +1  \quad \mod p.$$
must be congruent to zero. 
\end{proof}

We write 
$n=n' (q-1)+r$ with $0 \leq r <q-1$ and put $c_r=u^{-r} p^{-[pr/(q-1)]}  \varpi_p^r$. Then 
$$\varpi_p^n
=c_r p^{[pn/(q-1)]} u^n .$$
Hence we have 
$$L(n)=c_r \left( \frac{u}{\gamma}\right)^n\frac{c(n)}{\varpi^n_p}.$$
Therefore by Corollary \ref{corollary; c(n)} i), we have 
$|L(n)|<p$. (Note that if $n \not\equiv 0 \mod q-1$, then  $|c_r|<1$).  
Since $L(n)$ is contained in the unramified 
field $K$,  we have $L(n) \in \cO_K$. 
Similary, for $m \equiv n \mod p^l(q-1)$,  the fact $L(n) \in \cO_K$, 
Lemma \ref{lemma; ugamma}  and  Corollary \ref{corollary; c(n)} ii) 
imply  the congruence  
$$L(m)  \equiv {L(n)} \left(\frac{u}{\gamma}\right)^{m-n}  \equiv L(n) \quad \mod p^{l-\frac{p}{q-1}}.$$
Since this is a congruence between elements of $\cO_K$, we have 
$$L(m) \equiv L(n) \quad \mod p^{l}.$$
Similarly,  from Corollary \ref{corollary; c(n)} we obtain the congruences
originally proved by
Katz \cite[Theorem 3.1]{Ka2} and Chellali \cite[Th\'eor\`em 1.1]{Ch}.

\begin{theorem}\label{theorem, KC}
i) We have $L(n) \in \cO_K$. \\
ii) Suppose that $m \equiv n \mod p^l(q-1)$. Then 
$$L(m) \equiv L(n) \quad \mod p^{l}.$$
Furthermore, if $n \not\equiv 0 \mod q-1$, then 
$$L(m) \equiv L(n) \quad \mod p^{l+1}.$$
If  $n \equiv 0 \mod q-1$, then 
$$L'(m) \equiv L'(n) \quad \mod p^{l+1}.$$
\end{theorem}



\end{document}